\documentclass[12pt]{amsart}
\usepackage{amsmath,amsthm,amssymb}

\newcommand{\nexteq}{\displaybreak[0]\\ &=}
\newcommand{\nnext}{\displaybreak[0]\\ &}

\newtheorem{lem}{Lemma}

\newtheorem{prop}[lem]{Proposition}
\theoremstyle{definition}

\newtheorem{rmk}[lem]{Remark}

\newcommand{\CC}{\mathcal{C}}

\begin{document}
\title[Generalization of Knuth's formula]{Generalization of Knuth's formula for the number of skew tableaux}

\author{Minwon Na}
\address{Research Center for Pure and Applied Mathematics,
Graduate School of Information Sciences,
Tohoku University, Sendai 980--8579, Japan}
\email{minwon@ims.is.tohoku.ac.jp}

\date{January 2, 2016}

\begin{abstract}
We take an elementary approach to derive a generalization of Kunth's formula using Lassalle's explicit formula. In particular, we give a formula for the Kostka numbers of a shape $\mu\vdash n$ and weight $(m,1^{n-m})$ for $m=3,\;4$.
\end{abstract}

\maketitle

\section{Introduction}

Throughout this paper, $n$ will denote a positive integer. We write
$\mu\vdash n$ if $\mu$ is a partition of $n$, that is, a
non-increasing sequence $\mu=(\mu_1,\mu_2,\dots,\mu_k)$ of positive
integers such that $|\mu|=\sum_{i=1}^k\mu_i=n$. We say that $k$ is
the height of $\mu$ and denote it by $h(\mu)$. We denote by
$D_{\mu}$ the Young diagram of $\mu$. If
$\lambda=(\lambda_1,\lambda_2,\ldots,\lambda_h)\vdash m$ and
$D_{\lambda}\subset D_{\mu}$, then the skew shape $\mu/\lambda$ is
obtained by removing from $D_{\mu}$ all the boxes belonging to
$D_\lambda$.

Let $\mu,\;\lambda\vdash n$ and $\nu\vdash m\leq n$. A semistandard
Young tableau (SSYT) of shape $\mu$ and weight $\lambda$ is a
filling of the Young diagram $D_{\mu}$ with the numbers
$1,2,\ldots,h(\lambda)$ in such a way that
\begin{enumerate}
\item $i$ occupies $\lambda_i$ boxes, for $i=1,2,\ldots,h(\lambda)$,
\item the numbers are strictly increasing down the columns and weakly increasing along the rows.
\end{enumerate}
The Kostka number $K(\mu,\lambda)$ is the number of SSYTs of shape
$\mu$ and weight $\lambda$. In particular, if $\lambda=(1^n)$ then
such a tableau is called a standard Young tableau (SYT) of shape
$\mu$, and for a skew shape $\mu/\nu$ and weight $(1^{n-m})$ such a
tableau is called a skew SYT of skew shape $\mu/\nu$. We denote by
$f^{\mu/\nu}$ the number of skew SYTs of skew shape $\mu/\nu$.
Obviously, if $\lambda=(m,1^{n-m})\vdash n$ and $m\leq\mu_1$, then for all SSYTs of shape $\mu$ and weight $\lambda$, a box $(1,j)\in D_{\mu}$ is filled by $1$ for $1\leq j\leq m$, 
so $K(\mu,(m,1^{n-m}))=f^{\mu/(m)}$. Naturally, if $\nu=\emptyset$
then $f^{\mu}$ is the number of SYTs of shape $\mu$. We can easily
compute $f^{\mu}$ using the hook formula (see \cite{SF}), but the
problem of computing Kostka numbers is in general difficult (see
\cite{N}). There is a recurrence formula for Kostka numbers (see
\cite{L} and \cite{M}), but we have no explicit formula for Kostka
numbers.

For $z\in\mathbb{C}$, the falling factorial is defined by
$[z]_n=z(z-1)\cdots(z-n+1)=n!\binom{z}{n}$, and $[z]_0=1$. Let
$\mu=(\mu_1,\mu_2,\ldots,\mu_k)\vdash n$ and $\mu'$ be the conjugate
of $\mu$. Knuth \cite[p.67, Exercise 19]{K} shows:
\begin{equation}\label{sec1.eq.1}
f^{\mu/(2)}=\frac{f^{\mu}}{[n]_2}\left(\sum_{i=1}^k\binom{\mu_i}{2}
-\sum_{j\geq 1}\binom{\mu'_j}{2}+\binom{n}{2}\right).
\end{equation}
In fact, we can also compute $f^{\mu/\lambda}$ using
\cite[p.310]{A}, \cite[Theorem]{WF} and \cite[Corollary 7.16.3]{SE}, but this requires evaluation of determinants and
knowledge of Schur functions. If we compute $\lambda=(2)$ using
\cite[Corollary 7.16.3]{SE}, then we get the following:
\begin{equation}\label{sec1.eq.8}
f^{\mu/(2)}=\frac{f^{\mu}}{[n]_2}\left(\sum_{i=1}^k\left(\binom{\mu_i}{2}-\mu_i(i-1)\right)+\binom{n}{2}\right).
\end{equation}
Since the following equation is well known (see \cite[(1.6)]{M},
also see Proposition~\ref{prop:2} for a generalization):
\begin{equation}\label{sec1.eq.9}
\sum_{i=1}^k\mu_i(i-1)=\sum_{j\geq 1}\binom{\mu'_j}{2},
\end{equation}
we have (\ref{sec1.eq.1}). As previously stated, since $K(\mu,(m,1^{n-m}))=f^{\mu/(m)}$, we know the value of
$K(\mu,(2,1^{n-2}))$ from (\ref{sec1.eq.1}), so we are interested in
the extent to which (\ref{sec1.eq.1}) can be generalized to an arbitrary
positive integer $m$. In fact, if $\lambda=(3)$ then we get
the following using \cite[Corollary 7.16.3]{SE}:
\begin{align}\label{sec1.eq.7}
f^{\mu/(3)}&=\frac{f^{\mu}}{[n]_3}\left(
\sum_{i=1}^k\left(\mu_i(i-1)+\binom{\mu_i}{2}\right)
+(n-2)\sum_{i=1}^k\left(\binom{\mu_i}{2}-\mu_i(i-1)\right) \right)
\notag\nnext +\frac{f^{\mu}}{[n]_3}\left(
2\sum_{i=1}^k\left(\mu_i\binom{i-1}{2}+\binom{\mu_i}{3}\right)
-2\sum_{i=1}^k\binom{\mu_i}{2}(i-1) +\binom{n}{3}-\binom{n}{2}
\right).
\end{align}
The proof of (\ref{sec1.eq.7}) using Lassalle's explicit formula for characters will be given in Section~4.

Let $l$ be a nonnegative integer. Let $C(\mu)=\{j-i\mid (i,j)\in
D_{\mu}\}$ be the multiset of contents of the partition $\mu$, and
\[
p_l[C(\mu)]=\sum_{(i,j)\in D_{\mu}}(j-i)^l
\]
be the $l$th power sum symmetric function evaluated at the contents
of $\mu$. In this paper, we
take an elementary approach to derive a formula for $f^{\mu/(m)}$
using {\cite[Section 5.3]{C}} and $p_l[C(\mu)]$.

This paper is organized as follows. After giving preliminaries in
Section 2, we prove that $p_l[C(\mu)]$ can be written as a linear
combination of $q^{\pm}_{r,t}$ in Section 3. We give an
expression for $f^{\mu/(m)}$ in terms of $q^{\pm}_{r,t}$ for $m\leq
4$ in Section 4. Finally, we prove a generalization of (\ref{sec1.eq.9}) in Section 5.



\section{Preliminaries}
Throughout this section, $h,\;l,\;r$ and $t$ be
nonnegative integers. We denote by $S(n,k)$ the Stirling numbers of the second kind.
First of all, we define
\[
\mathcal{C}(r,t)=t!S(r+1,t+1).
\]
Then
\begin{align}\label{sec3.eq.1}
\CC(r,t)&=t!S(r+1,t+1)\notag\nexteq
t!(S(r,t)+(t+1)S(r,t+1))\notag\nexteq t\CC(r-1,t-1)+(t+1)\CC(r-1,t),
\end{align}
since $S(r+1,t+1)=S(r,t)+(t+1)S(r,t+1)$.

Set
\begin{equation}\label{sec3.eq.2}
\varphi_l(h,r,t)=\binom{l}{h}\CC(h,r)\CC(l-h,t).
\end{equation}
Clearly,
\begin{align}\label{sec3.eq.3}
\varphi_l(h,r,t)&= \binom{l}{l-h}\CC(l-h,t)\CC(h,r)\notag\nexteq
\varphi_l(l-h,t,r).
\end{align}

We define
\[
R_l(t)=\sum_{i=1}^ti^l.
\]

\begin{lem}\label{lem:2}
We have
\[
R_{l+1}(t)=(t+1)R_l(t)-\sum_{i=1}^tR_l(i).
\]
\end{lem}
\begin{proof}
We have
\begin{align*}
(t+1)R_l(t)&=(t+1)\sum_{i=1}^ti^l\nexteq
\sum_{i=1}^ti^{l+1}+\sum_{i=1}^t\sum_{j=1}^ij^l\nexteq
R_{l+1}(t)+\sum_{i=1}^tR_l(i).
\end{align*}
\end{proof}

\begin{lem}\label{lem:5}
We have
\[
R_l(t)=\sum_{i=0}^{l}\mathcal{C}(l,i)\binom{t}{i+1}.
\]
\end{lem}
\begin{proof}
Setting $n=q=0$ in \cite[Proposition 5.1.2]{C}. We have
\begin{equation}\label{lem:5.1}
\sum_{k=0}^l\binom{k}{m}=\binom{l+1}{m+1}.
\end{equation}

We prove the statement by induction on $l$. If $l=0$, then the statement holds since $\mathcal{C}(0,0)=1$.
Assume that the statement holds for $l-1$. Then
\begin{align*}
R_{l}(t)&=(t+1)R_{l-1}(t)-\sum_{j=1}^tR_{l-1}(j)&&\text{(by
Lemma~\ref{lem:2})}\nexteq(t+1)\sum_{i=0}^{l-1}\CC(l-1,i)\binom{t}{i+1}-\sum_{j=1}^t\sum_{i=0}^{l-1}\CC(l-1,i)\binom{j}{i+1}\nexteq
\sum_{i=0}^{l-1}(i+2)\CC(l-1,i)\binom{t+1}{i+2}-\sum_{i=0}^{l-1}\CC(l-1,i)\binom{t+1}{i+2}&&\text{(by
(\ref{lem:5.1}))}\nexteq
\sum_{i=0}^{l-1}(i+1)\CC(l-1,i)\binom{t+1}{i+2}\nexteq
\sum_{i=0}^{l-1}(i+1)\CC(l-1,i)\binom{t}{i+2}+\sum_{i=0}^{l-1}(i+1)\CC(l-1,i)\binom{t}{i+1}\nexteq
\sum_{i=1}^{l}i\CC(l-1,i-1)\binom{t}{i+1}+\sum_{i=0}^{l-1}(i+1)\CC(l-1,i)\binom{t}{i+1}\nexteq
\sum_{i=0}^{l}\left(i\CC(l-1,i-1)+(i+1)\CC(l-1,i)\right)\binom{t}{i+1}\nexteq
\sum_{i=0}^{l}\CC(l,i)\binom{t}{i+1}&&\text{(by (\ref{sec3.eq.1}))}.
\end{align*}
\end{proof}

\begin{lem}\label{lem:4}
For $z\in\mathbb{C}$, we have
\[
z^l=\sum_{i=0}^l\CC(l,i)\binom{z-1}{i}.
\]
\end{lem}
\begin{proof}
From \cite[p.211, (4.65)]{C}, we have
\[
z^l=\sum_{i=0}^lS(l,i)[z]_i,
\]
so
\begin{align*}
z^l&=\sum_{i=0}^lS(l,i)[z]_i\notag\nexteq
\sum_{i=0}^lS(l,i)z[z-1]_{i-1}\notag\nexteq
\sum_{i=0}^lS(l,i)[z-1]_{i-1}(z-i+i)\notag\nexteq
\sum_{i=0}^lS(l,i)[z-1]_{i}+\sum_{i=1}^liS(l,i)[z-1]_{i-1}\notag\nexteq
\sum_{i=0}^lS(l,i)[z-1]_{i}+\sum_{i=0}^{l-1}(i+1)S(l,i+1)[z-1]_{i}\notag\nexteq
\sum_{i=0}^l\left(S(l,i)+(i+1)S(l,i+1)\right)[z-1]_{i}\notag\nexteq
\sum_{i=0}^lS(l+1,i+1)[z-1]_{i}\notag\nexteq
\sum_{i=0}^li!S(l+1,i+1)\binom{z-1}{i}\notag\nexteq
\sum_{i=0}^l\CC(l,i)\binom{z-1}{i}.
\end{align*}
\end{proof}


Let $\mu$, $\lambda\vdash n$. We denote by $\chi^{\mu}(\lambda)$ the
value of the character of the Specht module $S^{\mu}$ evaluated at a
permutation $\pi$ belonging to the conjugacy class of type
$\lambda$.
%
%
%
From {\cite[Example 5.3.3]{C}}, we have
\begin{align}\label{sec2.eq.3}
&\chi^{\mu}(2,1^{n-2})=
\frac{f^{\mu}}{[n]_2}2p_1[C(\mu)],\notag\nnext
\chi^{\mu}(3,1^{n-3})=
\frac{f^{\mu}}{[n]_3}3\left(p_2[C(\mu)]-\binom{n}{2}\right),\notag\nnext
\chi^{\mu}(4,1^{n-4})=
\frac{f^{\mu}}{[n]_4}
4\left(p_3[C(\mu)]-(2n-3)p_1[C(\mu)]\right),\notag\nnext
\chi^{\mu}(5,1^{n-5})= \frac{f^{\mu}}{[n]_5}
5\left(p_4[C(\mu)]-(3n-10)p_2[C(\mu)]-2p_1[C(\mu)]^2+5\binom{n}{3}-3\binom{n}{2}\right),\notag\nnext
\chi^{\mu}(6,1^{n-6})=\frac{f^{\mu}}{[n]_6}
6\left(p_5[C(\mu)]+(25-4n)p_3[C(\mu)]+2(3n-4)(n-5)p_1[C(\mu)]\right)\notag\nnext
\qquad\qquad\qquad-\frac{f^{\mu}}{[n]_6}36p_1[C(\mu)]p_2[C(\mu)].
\end{align}

\begin{rmk}
In {\cite[Example 5.3.3]{C}}, the coefficient of $d_3(\lambda)$ (in
this paper, we denote by $p_3[C(\mu)]$) in the character value
$\hat{\chi}^{\lambda}_{6,1^{n-6}}$ is $24(7-n)$. Since
$c^{\lambda}_6$ and $c^{\lambda}_7$ are incorrect in
{\cite[p.251]{C}}, the value of the character
$\hat{\chi}^{\lambda}_{6,1^{n-6}}$ is also incorrect. In fact, the
coefficient of $d_3(\lambda)$ in the character value
$\hat{\chi}^{\lambda}_{6,1^{n-6}}$ is $6(25-4n)$, as given in (\ref{sec2.eq.3}).
\end{rmk}

We obtain \cite[Example 5.3.8]{C}:
\begin{align}\label{sec2.eq.4}
\chi^{\mu}(2,2,1^{n-4})&=
\frac{f^{\mu}}{[n]_4}4\left(p_1[C(\mu)]^2-3p_2[C(\mu)]+2\binom{n}{2}\right).
\end{align}

In general, for $\mu\vdash n$ and $\lambda\vdash m\leq n$, the
character $\chi^{\mu}(\lambda,1^{n-m})$ can be expressed as a
polynomial of $c^{\mu}_r(t)$ using Lassalle's explicit formula
\cite[Theorem 5.3.11]{C}.


\section{$p_l[C(\mu)]$ and $q^{\pm}_{r,t}$}
Let $\mu=(\mu_1,\mu_2,\ldots,\mu_k)\vdash
n$, and let $r,\;t$  be
nonnegative integers. We define
\begin{equation}\label{sec1.eq.2}
q^{\pm}_{r,t}=\sum_{i=1}^k\left(\binom{\mu_i}{r+1}\binom{i-1}{t}\pm
\binom{\mu_i}{t+1}\binom{i-1}{r}\right).
\end{equation}
Observe that if $r=t$
then
\begin{equation}\label{sec3.eq.6}
q^{-}_{r,r}=0,
\end{equation}
and
\begin{align}
q^{+}_{r,t}&=q^{+}_{t,r},\label{sec3.eq.5}\\
q^{-}_{r,t}&=-q^{-}_{t,r}.\label{sec3.eq.7}
\end{align}

%

\begin{prop}\label{prop:1}
Let $\mu=(\mu_1,\mu_2,\ldots,\mu_k)\vdash n$ and $l$ be a
nonnegative integer. Then
\begin{align*}
p_{2l+1}[C(\mu)]&=\sum_{h=0}^{l}\sum_{r=0}^{h}
\sum_{t=0}^{2l+1-h}(-1)^{h}\varphi_{2l+1}(h,r,t)q^{-}_{t,r},\\
p_{2l}[C(\mu)]&=\sum_{h=0}^{l-1}\sum_{r=0}^{h}
\sum_{t=0}^{2l-h}(-1)^{h}\varphi_{2l}(h,r,t)q^{+}_{r,t}
+\frac{1}{2}(-1)^{l}\sum_{r=0}^{l} \sum_{t=0}^{l}\varphi_{2l}(l,r,t)
q^{+}_{r,t}.
\end{align*}
\end{prop}

\begin{proof}
By the definition of $p_l[C(\mu)]$, we get the following:
\begin{align*}
p_l[C(\mu)]&= \sum_{i=1}^k\sum_{j=1}^{\mu_i}(j-i)^l\notag\nexteq
\sum_{i=1}^k\sum_{j=1}^{\mu_i}\sum_{h=0}^l(-1)^{l-h}\binom{l}{h}j^hi^{l-h}\notag\nexteq
\sum_{i=1}^k\sum_{h=0}^l(-1)^{l-h}\binom{l}{h}i^{l-h}R_h(\mu_i)\notag\nexteq
\sum_{i=1}^k\sum_{h=0}^l\sum_{r=0}^{h}\sum_{t=0}^{l-h}(-1)^{l-h}\binom{l}{h}\CC(h,r)\CC(l-h,t)\binom{\mu_i}{r+1}\binom{i-1}{t}
\notag \nexteq
\sum_{i=1}^k\sum_{h=0}^l\sum_{r=0}^{h}\sum_{t=0}^{l-h}(-1)^{l-h}\varphi_l(h,r,t)\binom{\mu_i}{r+1}\binom{i-1}{t}
&&\text{(by (\ref{sec3.eq.2}))},
\end{align*}
where the fourth equality follows from Lemma~\ref{lem:5} and
Lemma~\ref{lem:4}. Thus
\begin{align*}
p_{2l+1}[C(\mu)]&=\sum_{i=1}^k\sum_{h=0}^{2l+1}\sum_{r=0}^{h}
\sum_{t=0}^{2l+1-h}(-1)^{2l+1-h}\varphi_{2l+1}(h,r,t)
\binom{\mu_i}{r+1}\binom{i-1}{t}\nexteq
\sum_{i=1}^k\sum_{h=0}^{l}\sum_{r=0}^{h}
\sum_{t=0}^{2l+1-h}(-1)^{2l+1-h}\varphi_{2l+1}(h,r,t)
\binom{\mu_i}{r+1}\binom{i-1}{t}\nnext\quad+
\sum_{i=1}^k\sum_{h=l+1}^{2l+1}\sum_{r=0}^{h}
\sum_{t=0}^{2l+1-h}(-1)^{2l+1-h}\varphi_{2l+1}(h,r,t)
\binom{\mu_i}{r+1}\binom{i-1}{t}\nexteq
\sum_{i=1}^k\sum_{h=0}^{l}\sum_{r=0}^{h}
\sum_{t=0}^{2l+1-h}(-1)^{h-1}\varphi_{2l+1}(h,r,t)
\binom{\mu_i}{r+1}\binom{i-1}{t}\nnext\quad
+\sum_{i=1}^k\sum_{h=0}^{l}\sum_{r=0}^{h}
\sum_{t=0}^{2l+1-h}(-1)^{h}\varphi_{2l+1}(h,r,t)
\binom{\mu_i}{t+1}\binom{i-1}{r}\nexteq \sum_{h=0}^{l}\sum_{r=0}^{h}
\sum_{t=0}^{2l+1-h}(-1)^{h}\varphi_{2l+1}(h,r,t) \nnext\quad\quad
\cdot\left\{\sum_{i=1}^k\binom{\mu_i}{t+1}\binom{i-1}{r}-
\sum_{i=1}^k\binom{\mu_i}{r+1}\binom{i-1}{t}\right\}\nexteq
\sum_{h=0}^{l}\sum_{r=0}^{h}
\sum_{t=0}^{2l+1-h}(-1)^{h}\varphi_{2l+1}(h,r,t)q^{-}_{t,r},
\end{align*}
where the third equality can be shown as follows:
\begin{align*}
\sum_{h=l+1}^{2l+1}&\sum_{r=0}^{h}
\sum_{t=0}^{2l+1-h}(-1)^{2l+1-h}\varphi_{2l+1}(h,r,t)
\binom{\mu_i}{r+1}\binom{i-1}{t}\nexteq
\sum_{h=0}^{l}\sum_{r=0}^{2l+1-h}
\sum_{t=0}^{h}(-1)^{h}\varphi_{2l+1}(2l+1-h,r,t)
\binom{\mu_i}{r+1}\binom{i-1}{t}\nexteq
\sum_{h=0}^{l}\sum_{r=0}^{2l+1-h}
\sum_{t=0}^{h}(-1)^{h}\varphi_{2l+1}(h,t,r)
\binom{\mu_i}{r+1}\binom{i-1}{t}&&\text{(by (\ref{sec3.eq.3}))}
\nexteq \sum_{h=0}^{l}\sum_{r=0}^{h}\sum_{t=0}^{2l+1-h}
(-1)^{h}\varphi_{2l+1}(h,r,t) \binom{\mu_i}{t+1}\binom{i-1}{r}.
\end{align*}

Similarly, we have
\begin{align*}
p_{2l}[C(\mu)]&=\sum_{h=0}^{l-1}\sum_{r=0}^{h}
\sum_{t=0}^{2l-h}(-1)^{h}\varphi_{2l}(h,r,t)q^{+}_{r,t}
\nnext\quad+\sum_{i=1}^k\sum_{r=0}^{l}
\sum_{t=0}^{l}(-1)^{l}\varphi_{2l}(l,r,t)
\binom{\mu_i}{r+1}\binom{i-1}{t} \nexteq
\sum_{h=0}^{l-1}\sum_{r=0}^{h}
\sum_{t=0}^{2l-h}(-1)^{h}\varphi_{2l}(h,r,t)q^{+}_{r,t}
+\frac{1}{2}(-1)^{l}\sum_{r=0}^{l} \sum_{t=0}^{l}\varphi_{2l}(l,r,t)
q^{+}_{r,t},
\end{align*}
where the second equality can be shown as follows:
\begin{align*}
\sum_{i=1}^k&\sum_{r=0}^{l}\sum_{t=0}^{l}(-1)^{l}\varphi_{2l}(l,r,t)
\binom{\mu_i}{r+1}\binom{i-1}{t}\nexteq\frac{1}{2}(-1)^{l}
\sum_{i=1}^k\sum_{r=0}^{l}\sum_{t=0}^{l}\varphi_{2l}(l,r,t)
\binom{\mu_i}{r+1}\binom{i-1}{t}\nnext\quad +\frac{1}{2}(-1)^{l}
\sum_{i=1}^k\sum_{r=0}^{l}\sum_{t=0}^{l}\varphi_{2l}(l,r,t)
\binom{\mu_i}{t+1}\binom{i-1}{r}\nexteq \frac{1}{2}(-1)^{l} \sum_{r=0}^{l}
\sum_{t=0}^{l}\varphi_{2l}(l,r,t)q^+_{r,t}.
\end{align*}
\end{proof}

By Proposition~\ref{prop:1}, we have
\begin{align}\label{sec3.eq.8}
p_0[C(\mu)]&=\frac{1}{2}q^+_{0,0}=n,\notag\\
p_1[C(\mu)]&=q^-_{0,0}+q^-_{1,0}\notag\nexteq
q^-_{1,0},&&\text{(by (\ref{sec3.eq.6}))}\notag\\
p_2[C(\mu)]&=2q^+_{0,1}+2q^+_{0,2}-q^+_{1,0}-q^+_{1,1}\notag
\nexteq q^+_{0,1}+2q^+_{0,2}-q^+_{1,1},&&\text{(by (\ref{sec3.eq.5}))}\notag\\
p_3[C(\mu)]&=-2q^-_{1,0}+6q^-_{2,0}+6q^-_{3,0}-3q^-_{0,1}-9q^-_{1,1}-6q^-_{2,1}\notag\nexteq
q^-_{1,0}+6q^-_{2,0}+6q^-_{3,0}-6q^-_{2,1}&&\text{(by
(\ref{sec3.eq.6}) and (\ref{sec3.eq.7}))}.
\end{align}


\section{Main results}
For any $i\geq 1$, $m_i(\mu)=|\{j\mid \mu_j=i\}|$ is the
multiplicity of $i$ in $\mu$. Set
\[
z_{\mu}=\prod_{i\geq 1}i^{m_i(\mu)}m_i(\mu)!.
\]
Let $\mu\vdash n$ and $\lambda\vdash m\leq n$.
From \cite[Theorem 3.1]{S}, we have
\[
f^{\mu/\lambda}=\sum_{\nu\vdash
m}z^{-1}_{\nu}\chi^{\mu}(\nu,1^{n-m})\chi^{\lambda}(\nu).
\]
If $\lambda=(m)$, then
\begin{align}\label{sec4.eq.1}
f^{\mu/(m)}&=\sum_{\nu\vdash
m}z^{-1}_{\nu}\chi^{\mu}(\nu,1^{n-m})\chi^{(m)}(\nu)\notag\nexteq
\sum_{\nu\vdash m}z_{\nu}^{-1}\chi^{\mu}(\nu,1^{n-m}).
\end{align}

We already proved that
$p_l[C(\mu)]$
can be expressed as a linear combination of $q^{\pm}_{r,t}$
(Proposition~\ref{prop:1}), so the character value
$\chi^{\mu}(\lambda,1^{n-m})$ can be written as a polynomial in
$q^{\pm}_{r,t}$ using Lassalle's explicit formula \cite[Theorem
5.3.11]{C}. We compute $\chi^{\mu}(m,1^{n-m})$ for $2\leq m\leq 4$
and $\chi^{\mu}(2,2,1^{n-4})$ using (\ref{sec2.eq.3}),
(\ref{sec2.eq.4}) and (\ref{sec3.eq.8}).
\begin{align}\label{sec4.eq.2}
\chi^{\mu}(2,1^{n-2})&=
\frac{f^{\mu}}{[n]_2}2p_1[C(\mu)]\notag\nexteq
\frac{f^{\mu}}{[n]_2}2q^{-}_{1,0},\notag\\
\chi^{\mu}(3,1^{n-3})&=
\frac{f^{\mu}}{[n]_3}3\left(p_2[C(\mu)]-\binom{n}{2}\right)\notag\nexteq
\frac{f^{\mu}}{[n]_3}3\left(q^+_{0,1}+2q^+_{0,2}-q^+_{1,1}-\binom{n}{2}\right),\notag\\
\chi^{\mu}(4,1^{n-4})&=
\frac{f^{\mu}}{[n]_4}
4\left(p_3[C(\mu)]-(2n-3)p_1[C(\mu)]\right)\notag\nexteq
\frac{f^{\mu}}{[n]_4} 4\left((4-2n)q^-_{1,0}+6q^-_{2,0}+6q^-_{3,0}-6q^-_{2,1}\right),\notag\\
\chi^{\mu}(2,2,1^{n-4})&=
\frac{f^{\mu}}{[n]_4}4\left(p_1[C(\mu)]^2-3p_2[C(\mu)]+2\binom{n}{2}\right)\notag\nexteq
\frac{f^{\mu}}{[n]_4}4\left((q^-_{1,0})^2-3q^+_{0,1}-6q^+_{0,2}+3q^+_{1,1}+2\binom{n}{2}\right).
\end{align}
Substituting (\ref{sec4.eq.2}) into (\ref{sec4.eq.1}), we find
\begin{align}
f^{\mu/(2)}&=
\frac{1}{z_{(2)}}\chi^{\mu}(2,1^{n-2})+\frac{1}{z_{(1,1)}}\chi^{\mu}(1^n)
\notag\nexteq\frac{1}{2}\frac{f^{\mu}}{[n]_2}\cdot2q^{-}_{1,0}+\frac{1}{2}f^{\mu}\notag\nexteq
\frac{f^{\mu}}{[n]_2}\left(q^-_{1,0}+\binom{n}{2}\right),\label{sec4.eq.4}\\
f^{\mu/(3)}&=
\frac{1}{z_{(3)}}\chi^{\mu}(3,1^{n-3})+\frac{1}{z_{(2,1)}}\chi^{\mu}(2,1^{n-2})+\frac{1}{z_{(1,1,1)}}\chi^{\mu}(1^n)
\notag\nexteq\frac{1}{3}\frac{f^{\mu}}{[n]_3}\cdot3\left(q^+_{0,1}+2q^+_{0,2}-q^+_{1,1}-\binom{n}{2}\right)
+\frac{1}{2}\frac{f^{\mu}}{[n]_2}\cdot2q^{-}_{1,0}+\frac{1}{6}f^{\mu}\notag\nexteq
\frac{f^{\mu}}{[n]_3}\left(q^+_{0,1}+2q^+_{0,2}-q^+_{1,1}+(n-2)q^{-}_{1,0}+\binom{n}{3}-\binom{n}{2}\right),\label{sec4.eq.3}
\end{align}
and
\begin{align*}
f^{\mu/(4)}&=
\frac{1}{z_{(4)}}\chi^{\mu}(4,1^{n-4})+\frac{1}{z_{(3,1)}}\chi^{\mu}(3,1^{n-3})+\frac{1}{z_{(2,2)}}\chi^{\mu}(2,2,1^{n-4})\notag\nnext
\quad+
\frac{1}{z_{(2,1,1)}}\chi^{\mu}(2,1^{n-2})+\frac{1}{z_{(1,1,1,1)}}\chi^{\mu}(1^n)
\notag\nexteq \frac{1}{4}\frac{f^{\mu}}{[n]_4}\cdot
4\left((4-2n)q^-_{1,0}+6q^-_{2,0}+6q^-_{3,0}-6q^-_{2,1}\right)\notag\nnext
\quad+\frac{1}{3}\frac{f^{\mu}}{[n]_3}\cdot3\left(q^+_{0,1}+2q^+_{0,2}-q^+_{1,1}-\binom{n}{2}\right)\notag\nnext
\quad+\frac{1}{8}\frac{f^{\mu}}{[n]_4}\cdot4\left((q^-_{1,0})^2-3q^+_{0,1}-6q^+_{0,2}+3q^+_{1,1}+2\binom{n}{2}\right)\notag\nnext
\quad+\frac{1}{4}\frac{f^{\mu}}{[n]_2}\cdot2q^{-}_{1,0}+\frac{1}{24}f^{\mu}\notag\nexteq
\frac{f^{\mu}}{[n]_4}\left(\frac{1}{2}(n-2)(n-7)q^-_{1,0}+6q^-_{2,0}+6q^-_{3,0}-6q^-_{2,1}+\frac{1}{2}(q^-_{1,0})^2\right)\notag\nnext
\quad+\frac{f^{\mu}}{[n]_4}\left((n-\frac{9}{2})q^{+}_{0,1}+(2n-9)q^{+}_{0,2}-(n-\frac{9}{2})q^{+}_{1,1}\right)\notag\nnext
\quad+\frac{f^{\mu}}{[n]_4}\left(\binom{n}{4}-3\binom{n}{3}+2\binom{n}{2}\right).
\end{align*}
We get (\ref{sec1.eq.8}) and (\ref{sec1.eq.7}) by substituting (\ref{sec1.eq.2}) into (\ref{sec4.eq.4}) and (\ref{sec4.eq.3}), respectively.

\section{A generalization of a polynomial identity for a partition and its conjugate}

\begin{prop}\label{prop:2}
Let $\mu$ be a partition of an integer. Then $\mu'$ is the
conjugate of $\mu$ if and only if
\[
\sum_{i=1}^k\binom{\mu_i}{t+1}\binom{i-1}{r}=\sum_{j\geq1}\binom{\mu'_j}{r+1}\binom{j-1}{t}.
\]
for all nonnegative integers $r$ and $t$.
\end{prop}
\begin{proof} First, we show the ``only if'' part. Then
\begin{align*}
\sum_{j\geq1}\binom{\mu'_j}{r+1}\binom{j-1}{t}&=\sum_{j\geq t+1}
\sum_{\substack{J\subseteq\{1,2,\ldots,\mu_1\},\\|J|=t+1,\\ \max
J=j}}|\{I\mid I\times J\subseteq D_{\mu},\;|I|=r+1\}|\nexteq
\sum_{i=r+1}^{k}
\sum_{\substack{I\subseteq\{1,2,\ldots,k\},\\|I|=r+1,\\ \max
I=i}}|\{J\mid I\times J\subseteq D_{\mu},\;|J|=t+1\}|\nexteq
\sum_{i=r+1}^{k}
\sum_{\substack{I\subseteq\{1,2,\ldots,k\},\\|I|=r+1,\\ \max
I=i}}|\{J\mid \max J\leq\mu_i,\;|J|=t+1\}|\nexteq \sum_{i=r+1}^{k}
\sum_{\substack{I\subseteq\{1,2,\ldots,k\},\\|I|=r+1,\\ \max
I=i}}|\{J\mid J\subseteq\{1,2,\ldots,\mu_i\},\;|J|=t+1\}|\nexteq
\sum_{i=r+1}^{k}
\sum_{\substack{I\subseteq\{1,2,\ldots,k\},\\|I|=r+1,\\ \max
I=i}}\binom{\mu_i}{t+1}\nexteq\sum_{i=r+1}^{k}\binom{\mu_i}{t+1}\binom{i-1}{r}.
\end{align*}

Next, let $\lambda$ be the conjugate of $\mu$. Set $h(\lambda)=h$.
Then
\begin{align}\label{prop:2.1}
\sum_{j=1}^h\binom{\lambda_j}{r+1}\binom{j-1}{t}&=
\sum_{i=1}^k\binom{\mu_i}{t+1}\binom{i-1}{r}\notag\nexteq
\sum_{j\geq1}\binom{\mu'_j}{r+1}\binom{j-1}{t}.
\end{align}
Setting $h(\mu')=l$ and $r=0$ in (\ref{prop:2.1}), we have
\begin{equation}\label{prop:2.2}
\sum_{j=1}^h\lambda_j\binom{j-1}{t}=
\sum_{i=1}^l\mu'_i\binom{i-1}{t}.
\end{equation}
Suppose $h>l$ and set $t=h-1$ in (\ref{prop:2.2}), then
$\lambda_h=0$. Similarly, suppose $h<l$ and set $t=l-1$ in
(\ref{prop:2.2}). Then $\mu'_l=0$, and both cases are contradictions.
Thus $h=l$.

We show that $\lambda_{h-i}=\mu'_{h-i}$ for all $i$ with $0\leq
i\leq h-1$ by induction on $i$. If $i=0$, setting $t=h-1$ in
(\ref{prop:2.2}), then $\lambda_h=\mu'_h$.

Assume that the assertion holds for some $i\in\{0,1,\ldots,h-2\}$.
Let $t=h-(i+2)$ in (\ref{prop:2.2}). By the inductive hypothesis, we
have \[
\sum_{j=h-i}^h\lambda_j\binom{j-1}{h-i-2}=\sum_{j=h-i}^h\mu'_j\binom{j-1}{h-i-2}.
\]
Therefore, $\lambda_{h-i-1}=\mu'_{h-i-1}$ since
$\binom{j-1}{h-i-2}=0$ for all $j$ with $1\leq j\leq h-j-2$. Thus
$\lambda=\mu'$ and $\mu'$ is the conjugate of $\mu$.
\end{proof}

From Proposition~\ref{prop:2}, we have
\begin{equation}\label{sec5.eq.1}
q^{\pm}_{r,t}=\sum_{i=1}^k\binom{\mu_i}{r+1}\binom{i-1}{t}\pm\sum_{j\geq
1}\binom{\mu'_j}{r+1}\binom{j-1}{t}.
\end{equation}
By substituting (\ref{sec5.eq.1}) into (\ref{sec4.eq.4}) and (\ref{sec4.eq.3}), we get (\ref{sec1.eq.1}) and
\begin{align*}
f^{\mu/(3)}&=\frac{f^{\mu}}{[n]_3}\left(q^+_{0,1}+2q^+_{0,2}-q^+_{1,1}+(n-2)q^{-}_{1,0}+\binom{n}{3}-\binom{n}{2}\right)\nexteq
\frac{f^{\mu}}{[n]_3}\left(q^+_{1,0}+2q^+_{2,0}-q^+_{1,1}+(n-2)q^{-}_{1,0}+\binom{n}{3}-\binom{n}{2}\right)
&&\text{(by
(\ref{sec3.eq.5}))}\nexteq
\frac{f^{\mu}}{[n]_3}\left(\left(
\sum_{i=1}^k\binom{\mu_i}{2}+\sum_{j\geq
1}\binom{\mu'_j}{2}\right)+2\left(
\sum_{i=1}^k\binom{\mu_i}{3}+\sum_{j\geq
1}\binom{\mu'_j}{3}\right)
\right)\nnext
\quad-\frac{f^{\mu}}{[n]_3}\left(
\sum_{i=1}^k\binom{\mu_i}{2}(i-1)+\sum_{j\geq
1}\binom{\mu'_j}{2}(j-1)
\right)\nnext
\quad+\frac{f^{\mu}}{[n]_3}\left(
(n-2)\left(\sum_{i=1}^k\binom{\mu_i}{2}-\sum_{j\geq
1}\binom{\mu'_j}{2}\right)+\binom{n}{3}-\binom{n}{2}
\right),
\end{align*}
respectively.

\end{document}